\documentclass[10pt]{amsart}
\usepackage[]{amsmath, amsthm, amsfonts,amssymb}
\usepackage{graphicx}
\usepackage[all]{xy}
 \usepackage{xcolor}
 \usepackage{hyperref}

\newcommand {\C} {{\mathbb C}}
\newcommand {\R} {{\mathbb R}}
\newcommand {\Z} {{\mathbb Z}}
\newcommand {\Q} {{\mathbb Q}}

\newcommand {\dt} {\bullet}
\newcommand {\F} {{\mathcal F}}

\newcommand {\OO} {{\mathcal O}}

\DeclareMathOperator{\im}{im}

\newtheorem{thm}[subsection]{Theorem}
\newtheorem{cor}[subsection]{Corollary}
\newtheorem{lemma}[subsection]{Lemma}
\newtheorem{prop}[subsection]{Proposition}

\newtheorem{rmk}[subsection]{Remark}
\newtheorem{ex}[subsection]{Example}

\newtheorem{conj}[subsection]{Conjecture}

\numberwithin{equation}{section}

\begin{document}


\title{Euler characteristics of Koll\'ar-hyperbolic varieties }

\author{Donu Arapura}
\thanks{Author supported by a grant from the Simons foundation}
\address{Department of Mathematics\\ Purdue University\\
West Lafayette, IN 47907-2067}
\email{arapura@purdue.edu}

\maketitle

\begin{abstract}
 Call a normal complex  projective variety $X$  Koll\'ar-hyperbolic   if   any nonconstant map from a  smooth projective curve to $X$ induces a nontrivial homomorphism of \'etale fundamental groups.  Examples include (a) smooth varieties with finite Albanese  map, (b) normalizations of subvarieties
 of hermitian locally symmetric varieties of noncompact type, and (c) higher dimensional Kodaira fibrations. We conjecture that
 Koll\'ar-hyperbolic varieties
satisfy a vanishing theorem, which says roughly that if $P$ is perverse sheaf underlying  a mixed Hodge module
on such a variety then the limit of  normalized  dimensions of the cohomology groups of $P$ are zero in nonzero degrees, where the limit is taken over a suitable
tower of \'etale covers. We call such varieties $V$-hyperbolic. $V$-hyperbolic varieties satisfy a Gromov type vanishing theorem for $L^2$ cohomology, the inequalities
$(-1)^d\chi(X) \ge 0$ and $(-1)^{d-p}\chi(\Omega_X^p)\ge 0$ in the smooth case, and more generally,  an  inequality for mixed Hodge modules conjectured under related assumptions by Maxim, Wang and the author. We prove that examples of type (a) and (c) listed above are $V$-hyperbolic.
\end{abstract}

In short,  the initial motivation for this  paper  was to understand  some deeper reasons for why a conjecture of  Maxim, Wang and the author \cite{amw} ought to be true. 
Before saying more about this,
we need to explain a few notions.
We will call a normal complex projective variety  $X$ {\em Koll\'ar-hyperbolic}   if   any nonconstant map from a  smooth projective curve to $X$ induces a non trivial homomorphism
of \'etale fundamental groups, i.e. of  profinite completions of the usual fundamental groups. This notion (in an equivalent formulation) was introduced by Koll\'ar \cite{kollar, kollar2} under the  name of  ``varieties with large algebraic fundamental group". We prefer the present term, because it is more succinct  and
 it also happens to sound similar to closely related term. The class of Koll\'ar-hyperbolic varieties is fairly natural, and it  includes:
\begin{enumerate}
 \item[(a)] Smooth projective varieties whose Albanese map is finite over its image, 
 \item[(b)] normalizations of subvarieties of hermitian locally symmetric spaces of noncompact type, and
 \item[(c)] Kodaira fibrations (defined inductively as projective manifolds admitting a smooth projective map to a curve of positive genus with Kodaira fibrations as fibres).
\end{enumerate}

 Gromov introduced the related class of  K\"ahler-hyperbolic  manifolds in \cite{gromov}.
 We will recall the definition in the second section. For the moment, we just point out that a K\"ahler-hyperbolic
 manifold with residually finite fundamental group  is Koll\'ar-hyperbolic. Gromov proved that the $L^2$-cohomology of the universal cover of a compact K\"ahler-hyperbolic manifold vanishes in degrees different from $d=\dim X$. When combined with the Hodge decomposition and Atiyah's index theorem \cite{atiyah}, this implies the inequality $(-1)^{d-p}\chi(\Omega_X^p)\ge 0$ above.
 Loosely inspired by Gromov's theorem, we define a normal projective variety to be $V$-hyperbolic (V for ``vanishing") if it satisfies the following condition:
 There exists a subgroup  $H\subset \pi_1(X)$ of infinite index such  that 
for  any  tower $\pi_n:X_n\to X$ of conected \'etale covers with $\bigcap \pi_1(X_n)=H$, and any perverse sheaf $P$ underlying a mixed Hodge module,
$$\lim_{n\to \infty} \frac{\dim H^i(X_n, \pi_n^*P)}{\deg \pi_n} = 0$$
 for $i \not= 0$.  When $X$ is a smooth  $V$-hyperbolic variety of dimension $d$, and the above condition is applied to  $P= \Q[d]$,  we easily  deduce that $L^2$-cohomology
vanishes in degrees different from $d$. The key point is that in this case, L\"uck \cite{luck} proved  the limit appearing in the theorem  is an  $L^2$-Betti number.
One can then obtain the inequality$(-1)^{d-p}\chi(\Omega_X^p)\ge 0$ as above.
However, we give a different argument which  shows more generally that
\begin{equation}
  \label{eq:AMV}
\chi(Gr^p_F DR(M))\ge 0  
\end{equation}
for any mixed Hodge module $M$.
This is  in accordance with a conjecture of \cite{amw}. This can be spelled out
 in the simplest case where $M$ arises from a variation of Hodge structure on the complement
of a normal crossing divisor $D$ with unipotent local monodromy about it.
If    $( V, F^\dt)$ is the  filtered vector bundle associated to the  VHS, \eqref{eq:AMV} is equivalent to 
$$ \sum_k (-1)^{d-k}\chi(\Omega_X^k(\log D)\otimes  Gr_{ F}^{p-k}  V)\ge 0$$
 This  refines  the Arakelov inequalities of Eyssidieux \cite{eyss}, who proved this when $D=0$.
 We also obtain a topological corollary that if $P$ is a perverse sheaf associated to a Hodge module on a  Koll\'ar-hyperbolic variety, then  $\chi(X, P)\ge 0$.  
 In particular, $(-1)^d\chi(X)\ge 0$, when $X$ is smooth of dimension $d$. For related results in this direction, see the papers of
 Liu, Maxim, Wang \cite{lmw}, Wang and the author \cite{aw}, Patel and the author \cite{ap} and Deng and Wang \cite{dw}. However, these papers use  very different methods which do not imply  \eqref{eq:AMV}.


 We show that $V$-hyperbolic varieties are Koll\'ar-hyperbolic, and we conjecture the converse.
Evidence is provided by the following corollaries of our main theorems, which show that two of three classes
of Koll\'ar-hyperbolic varieties listed above are in fact $V$-hyperbolic.

\begin{cor}\label{cor:intro1}
 Smooth projective varieties whose Albanese map is finite over its image are V-hyperbolic.
\end{cor}

\begin{cor}\label{cor:intro2}
Generalized Kodaira fibrations are  V-hyperbolic. (These are defined inductively as a Koll\'ar-hyperbolic varieties $X$ which admit  a surjective morphism
$f:X\to C$ to a curve of positive genus, such that normalizations of irreducible components of all fibres are generalized Kodaira fibrations.)
\end{cor}

Corollary \ref{cor:intro1} is deduced from theorem \ref{thm:av}, that abelian varieties are $V$-hyperbolic, and  theorem \ref{thm:finVhyp}, that the class of $V$-hyperbolic varieties is stable under branched covers.  Theorem \ref{thm:av}  is reduced to a generic vanishing theorem of Popa and Schnell \cite{ps}.
Corollary \ref{cor:intro2} follows from theorem \ref{thm:key}, that  if $f:X\to C$ is a morphism of a normal projective variety onto a curve of positive genus such that all fibres are V-hyperbolic, then $X$
 is V-hyperbolic. This proved using the decomposition theorem and the stability of Hodge modules under vanishing cycle functors.

I would like to thank J\'anos Koll\'ar and Botong Wang for pointing out some issues with a preliminary version of this paper.
 
\section{Perverse sheaves and Hodge modules}

This section is purely expository, and briefly summarizes some facts about perverse sheaves and Hodge modules needed below.
In this paper, algebraic varieties are  complex algebraic varieties endowed  with the complex topology. These are  irreducible unless stated otherwise.

Let $X$ be an algebraic  variety, and let $D_c^b(X)$ denote the bounded  derived category of sheaves of  $\Q$-vector spaces with constructible
cohomology on $X$.
We recall that the category of perverse sheaves $Perv(X)\subset D_c^b(X)$
is the full subcategory  consisting of objects $P$ satisfying
\begin{equation}\label{eq:perv}
 \dim \operatorname{supp} \mathcal{H}^{-i}(P)\le i,\quad \dim \operatorname{supp} \mathcal{H}^{-i}(DP)\le i
\end{equation}
where $DP$ is the Verdier dual \cite{bbd, htt}.
This is an artinian abelian category with an involution given by $D$. Some basic examples are given below.

\begin{ex}
 If $X$ is smooth of dimension $d$, and $L$ is a local system such as $\Q_X$, then $L[d]\in Perv(X)$. More generally, if $L$ is irreducible and defined on the smooth locus $U = X- X_{sing}$, $L[d]$ has a unique extension to a simple object of $Perv(X)$. This is 
 the intersection cohomology complex  $IC(L)$
 \end{ex}

\begin{ex}
 Let $X$ be smooth of dimension $d$ and $M$ is a regular holonomic $D_X$-module \cite{htt}. We assume by default that $D$-modules are left modules.
 Then  the  de Rham complex 
 $$DR(M)= (M\to \Omega_X^1\otimes M\to\Omega_X^2\otimes M \ldots)[d]$$
 is a perverse sheaf of $\C$-vector spaces.
\end{ex}

Saito \cite{saito, saito2} (see also \cite{schnell})  gives a Hodge theoretic enrichment of $Perv(X)$ as follows.
First recall \cite{schmid} that a  variation of Hodge structure $H= (L, V,\nabla, F^\dt,\alpha)$  of weight $k$ consists of a rational local system $L$, a flat vector
bundle $(V, \nabla)$, with an isomorphism $\alpha:\C\otimes L \cong \ker \nabla$, and a decreasing filtration $F^\dt$ 
by subbundles of $V$. These are subject to the following axioms:
\begin{itemize}
 \item Griffiths transversality: $\nabla (F^p)\subset \Omega_X^1\otimes F^{p-1}$,
 \item  the  fibres $(L_x, V_x, F_x^\dt)$ are Hodge structures of weight $k$.
\end{itemize}
We also assume that $H$ carries an (unspecified) polarization.
In Saito's theory, the above data is replaced by a  quadruple $(P, M, F_\dt, \alpha)$, consisting of a perverse sheaf $P\in Perv(X)$, a  
regular holonomic $D$-module $(M,F)$ with a good filtration on an ambient smooth variety (the choice of which is immaterial) with an isomorphism $\alpha:\C\otimes P \cong DR(M)$. Saito constructs a suitable  semisimple abelian category $HM(X,w)$ (or $MH(X,w)$ {\em en francais}) of pure polarizable Hodge modules of weight $w$ on $X$. The precise axioms are complicated, but at the end of the day, Saito   \cite[thm 3.21]{saito2} gives the following description of the objects:

\begin{thm}[Saito]\label{thm:simpleHM}
 A simple object of $HM(X,w)$ corresponds to an irreducible   variation of Hodge structure of weight $w-\dim U$ on a smooth locally closed irreducible subvariety $U\subset X$.  More precisely, if $H= (L, V,\nabla, F^\dt, \alpha)$ is a VHS on $U\xrightarrow{j} \bar U\xrightarrow{\iota} X$, then
 the corresponding Hodge module $IC^H(H)=(P, M, F, \alpha)$ has $P = \iota_*j_{!*} L[\dim U]= \iota_*IC(L)$, and  $M$ is the minimal extension of
the $D_U$-module associated to  $(V,\nabla)$.  The filtration on $M$ is given by suitably extending the initial filtration on the VHS.
\end{thm}

We omit the precise description of the above filtered $D$-module $(M,F)$, since it somewhat involved. However, there is one case where  this can be
made fairly explicit \cite[\S 3]{saito2}.

\begin{ex}\label{ex:ncd}
Suppose that $X$ is smooth, and $D\subset X$ is a divisor with simple normal crossings. Let $H= (L, V,\nabla, F^\dt, \alpha)$ be a VHS
on $j:U= X-D\to X$. Let  $(V^{>-1},\bar  \nabla)$ be the Deligne extension of $(V,\nabla)$ to a bundle with a logarithmic connection with residues having
eigenvalues in the interval $(-1, 0]$. Then the  filtered  left $D_X$-module of $IC^H(H)$ is given by
$$M = D_X \cdot V^{>-1},\quad F_p M = \sum_i F_i D_X\cdot V^{>-1}\cap j_*F^{i-p} V $$
where $D_X$ acts on $V^{>-1}$ through $\bar \nabla$ and $F_\dt D_X$ is the filtration by order. We also note that the corresponding filtered right module is
$$ M^R = \omega_X\otimes M, \quad F_p M^R = \omega_X\otimes F_{p+\dim X} M$$
\end{ex}



If $X$ is smooth of dimension $d$  and $L=\Q_X$ is given the constant variation of Hodge structure of type $(0,0)$, then we also write $\Q^H[d] = IC^H(L)$. 

  A mixed Hodge module consists of  $(P, M, F, W, \alpha)$ where $W$ is an  additional
 filtration on $P$ and $M$, such that the associated graded object $W_k/W_{k-1}\in HM(X,k)$. This is subject to additional conditions that we will not describe.
 We will often just use $M$ to represent the whole mixed Hodge module.
 These form an abelian category $MHM(X)$.  Let 
 $$Perv_{HM}(X)\subset Perv_{MHM}(X)\subset Perv(X)$$
 denote the essential images under the forgetful functor from $\bigoplus_w HM(X,w)\subset MHM(X)$ to $Perv(X)$. These are both abelian and the first is semisimple.
 Given a mixed Hodge module $M$, the associated graded of the de Rham complex $Gr^p_FDR(M)= Gr_{-p}^F DR(M)$ is an object of the bounded coherent derived category $D^b_{coh}(X)$.
 We mention a few basic examples needed later.

\begin{ex}\label{ex:Qd}
 Let $X$ be smooth of dimension $d$, then 
 $$Gr_{F}^p DR(\Q_X^H[d]) = \Omega_X^p[d+p]$$
\end{ex}

\begin{ex}\label{ex:Vd}
 Let $X$ be smooth of dimension $d$, let $D\subset X$ be divisor with normal crossings, and $H = (L, V,\ldots)$  a polarizable variation of Hodge structure on $U=X-D\xrightarrow{j} X$. Then the perverse sheaf $\R j_* L[d]$ is part of a mixed Hodge module  $M$ on $X$. Suppose $L$ has unipotent local monodromy about $D$. Let $\bar V = V^{>-1}$ as in example \ref{ex:ncd}, and let  $\bar F^\dt= j_*F^\dt\cap \bar V\subseteq \bar V$.
 In the left $D$-module picture, the Hodge filtration on  $M$  satisfies
 $$Gr_{F}^p DR(M) \cong  (Gr_{\bar F}^p \bar V\xrightarrow{\theta} \Omega_X^1(\log D)\otimes Gr_{\bar F}^{p-1} \bar V\xrightarrow{\theta}  \Omega_X^2(\log D)\otimes Gr_{\bar F}^{p-2} \bar V\xrightarrow{\theta} \ldots)[d]$$
 where the ``Higgs field"  $\theta =  Gr_{\bar F}\bar  \nabla$. 
 \end{ex}

\begin{ex}\label{ex:kollar}
Let $X$ be an algebraic variety with a smooth Zariski open subset $j:U\to X$. Let  $H $ a polarizable variation of Hodge structure
on $U$. Choose $p=p_{\max}(H)$ to be the maximal $p$ for which $F^p\not= 0$. Then if use right $D$-modules in the set up,
$Gr_{F}^p DR(IC^H(H))$ is a torsion free coherent sheaf, denoted in \cite{saito3} by $S_X(H)$. 
(The torsion freeness is not explicitly stated in [loc. cit.], but it is easy to deduce by resolving singularities $\pi:\tilde X\to X$,
using $\pi_*S_{\tilde X}(H) = S_X(H)$, and using the formula described in example
\ref{ex:ncd} to see that this is torsion free.)
For example, if $f:Y\to X$ is a proper map, smooth over $U$, and $L= R^if_*\Q|_U$ (as a  VHS), then $S_X(L) = R^if_*\omega_Y$.
We dub $S_X(H)$ the canonical sheaf of $H$.
\end{ex}

Since the conditions \eqref{eq:perv} are local for the classical topology, given an \'etale map $\pi:Y\to X$,   the operation $\pi^*:D_c^b(X)\to D_c^b(Y)$ takes $Perv(X)\to Perv(Y)$.
Saito \cite[thm 0.1]{saito2} gives a functor $\pi^*:D^bMHM(Y)\to D^bMHM(X)$ compatible with the previous functor.
When $X$ is smooth, $\pi^* (P, M, F, W,\alpha) = (\pi^* P, \pi^*M,\ldots)$.

Finally, we observe that there is a generalization of the Hodge decomposition.

\begin{thm}[Saito]\label{thm:strictness}
 If $X$ is projective, and $(P,M, F,\ldots) \in MHM(X)$, then there is a noncanonical isomorphism
 $$H^i(X, P)\otimes \C \cong \bigoplus_p H^i(X,Gr_F^p DR(M))$$
\end{thm}

\begin{proof}
 See \cite[thm 28.1]{schnell}.
\end{proof}

Some additional facts about Hodge modules will be recalled later as we need them.

\section{Koll\'ar-hyperbolic varieties}

First we explain some notation and conventions used in this paper.
The word curve without qualification means smooth projective curve.
We will write $h^i = \dim H^i$.
If $f:X\to Y$ is a continuous  map of connected spaces, we denote the induced map  $\pi_1(X)\to \pi_1(Y)$ by $f$ as well.
Let  $\hat \Gamma$ denote the profinite completion of a group $\Gamma$, and  let $K(\Gamma) = \ker[\Gamma\to \hat \Gamma]$.
We will usually formulate our definitions and results using subgroups of $\hat \pi_1(X)$, since the statements are bit cleaner.
But readers who prefer to do so can translate these  to statements involving subgroups of $\pi_1(X)$, because there is a bijection between the set of (normal, finite index) subgroups $K(\Gamma)\subseteq G\subseteq \Gamma$ and the 
set of closed (normal, open) subgroups $H\subseteq \hat\Gamma$. In one direction $G\mapsto \bar G$ and in the other $H\mapsto H\cap \Gamma$.

Let $X$ be a normal projective variety with a closed normal subgroup $H\subset \hat\pi_1(X)$. 
We will say that $X$  {\em Koll\'ar-hyperbolic} if any nonconstant map from a smooth projective curve  $C\to X$ induces a nontrivial homomorphism  $\hat \pi_1(C)\to \hat \pi_1(X)$. 
It is useful to refine this a bit. Given a closed normal subgroup $H\subset \hat\pi_1(X)$, let us say that 
$X$  {\em Koll\'ar-hyperbolic }with respect to $H$, or that $(X,H)$ is      Koll\'ar-hyperbolic, if 
any nonconstant map from a smooth projective curve  $C\to X$ induces a nontrivial homomorphism  $\hat \pi_1(C)\to \hat \pi_1(X)/H$.  
 Clearly  $X$ is Koll\'ar-hyperbolic if $(X,H)$ is for some $H$.
 Koll\'ar \cite{kollar, kollar2} used a different condition, given in (b) of the next proposition, as the definition for what  he would call varieties with large algebraic fundamental group.
 If $G\subset \pi_1(X)$ corresponds to $H$ as in the previous paragraph, 
 $(X,H)$ is  Koll\'ar-hyperbolic if the Shafarevich map $Sh^G$, defined in \cite{kollar2},  is defined everywhere and an isomorphism.

\begin{prop}\label{prop:defkhyp}
Let $X$ be a normal  projective variety and $H\subset \hat\pi_1(X)$ a closed normal subgroup.
\begin{enumerate}
\item[(a)] $(X,H)$ is  Koll\'ar-hyperbolic  if and only if  for any nonconstant map from a smooth projective curve  $C\to X$, the image  $\im[\hat \pi_1(C)\to \hat \pi_1(X)/H]$ is infinite.
 \item[(b)] $(X,H)$ is  Koll\'ar-hyperbolic  if and only if for any subvariety $Y\subset X$, with normalization $Y^n$,
 $\im[\hat \pi_1(Y^n)\to \hat \pi_1(X)/H]$ is infinite.
\end{enumerate}

\end{prop}

\begin{proof}
  Suppose $(X,H)$ is   Koll\'ar-hyperbolic. Given a smooth curve $C$ and nonconstant map $f:C\to X$,
 suppose that $f(\hat \pi_1(C))$ is finite. Then we can find an \'etale cover $g:D\to C$ such that $f\circ g (\hat \pi_1(D))$ is trivial. This gives a contradiction. Therefore
  $f(\hat \pi_1(C))$ is infinite. The other direction for (a) is obvious.
  
  Suppose  that for any subvariety $Y\subset X$, with normalization $Y^n$,
 $\im[\hat \pi_1(Y^n)\to \hat \pi_1(X)/H]$ is infinite. Let $f:C\to X$ be a nonconstant map from a curve. Let $D = f(C)^n$. Then $\im [\hat \pi_1(C)\to \hat \pi_1(D)]$ has finite index in $\hat \pi_1(D)$.
 Therefore $f(\hat \pi_1(C))$ has finite index in the infinite group $\im [\hat \pi_1(D)\to \hat \pi_1(X)]$. Therefore $f(\hat \pi_1(C))$ is infinite.
 This proves that $(X,H)$ is Koll\'ar-hyperbolic. Conversely, suppose that $(X,H)$ is Koll\'ar-hyperbolic. Let $Y\subset X$ be a subvariety. Let $\pi:\tilde Y\to Y^n$ be a resolution of singularities.
 The Lefschetz theorem \cite{milnor} implies that we can find a smooth curve $C$, given as a complete intersection of ample divisors, and a surjection $\pi_1(C)\to \pi_1(\tilde Y)$.
 Moreover, we can assume by construction that $C$ maps nontrivially to $X$.
 The map $\pi$ is surjective with connected fibres, because $Y^n$ is normal. Therefore, we have a surjection $\pi_1(\tilde Y)\to \pi_1(Y^n)$. Thus we get a surjection
 $\hat \pi_1(C)\to \hat \pi_1(Y^n)$. Consequently $\im[\hat \pi_1(Y^n)\to \hat \pi_1(X)]$ coincides with $\im[\hat \pi_1(C)\to \hat \pi_1(X)]$, but the second group is infinite by (a).
  
\end{proof}

Gromov \cite{gromov} defines a complex manifold $X$ to be {\em K\"ahler-hyperbolic} if it  possesses a K\"ahler form $\omega$ such that $\pi^*\omega = d\alpha$, where $\pi:\tilde X\to X$ is the universal cover and $\|\alpha\|$ is bounded. A compact K\"ahler manifold with negative sectional curvature is K\"ahler-hyperbolic.
Other  examples and properties  can be found in \cite{gromov}.
The condition has been relaxed by Cao-Xavier \cite{cx},  Eyssidieux \cite{eyss} and Jost-Zuo \cite{jz} to allow nonpositively curved manifolds.
Specifically, the first and last paper defines  $X$ to be  {\em K\"ahler nonelliptic}  if $\pi^*\omega = d\alpha$ where $\|\alpha\|$ has sublinear growth. For example, an abelian variety satisfies the last
condition, but it is not K\"ahler-hyperbolic.

\begin{prop}\label{prop:khyp}
Let  $X, X_1, X_2$ be  normal projective varieties and let $H\subset \hat \pi_1(X)$, $H_i\subset \hat \pi_1(X_i)$  be closed normal subgroups.  Let $\tilde X_K= \tilde X/K(\pi_1(X))$.
\begin{enumerate}
 \item[(a)] If $X_1$ and $X_2$  are Koll\'ar-hyperbolic (with respect to $H_1$ and $H_2$), then $X_1\times X_2$  is  Koll\'ar-hyperbolic (with respect to $H_1\times H_2$).
 
  \item[(b)] If $X$ is Koll\'ar-hyperbolic (with respect to $H$) and $Y$ is normal and projective with a map $Y\to X$ which is  finite over its image, then $Y$ is   Koll\'ar-hyperbolic (with respect to the preimage of $H$).
 \item[(c)] If there is a representation $\rho:\pi_1(X)\to GL_n(F)$, for some field $F$, such that for any nonconstant map from a smooth projective curve  $f:C\to X$,
 $\rho\circ f:\pi_1(C)\to GL_n(F)$ has nontrivial image, then $X$ is  Koll\'ar-hyperbolic with respect to $\overline{\ker\rho}$.
 \item[(d)] If $X$ is smooth and if it possesses   a K\"aher form $\omega$ such that its pullback to $\tilde X_K$ is exact, then $X$ is Koll\'ar-hyperbolic.
 \item[(e)] If  $X$ is smooth, $\pi_1(X)$ is residually finite and   if either $X$ is K\"ahler nonelliptic or $\pi_2(X)=0$, then $X$ is Koll\'ar-hyperbolic. A K\"ahler-hyperbolic manifold with residually finite fundamental group is Koll\'ar-hyperbolic.
\end{enumerate}

\end{prop}
\begin{proof} A nontrivial map $C\to X_1\times X_2$ from a curve, must induce a nontrivial map on one of the factors. Therefore we
  have a nontrivial map 
 $$\hat \pi_1(C)\to \hat \pi_1(X_1\times X_2)/H_1\times H_2\cong \hat \pi_1(X_1)\times \hat \pi_1(X_2)/H_1\times H_2$$
  This proves (a).
Let $Y\to X$ be as in (b), and let $H'\subset \hat \pi_1(Y)$ be the preimage of $H$. 
Given a nonconstant map from a curve $C\to Y$, the composite $C\to X$ is nonconstant. Therefore $\im[\hat \pi_1(C)\to \hat \pi_1(X)/H]$ is nontrivial, which implies $\im[\hat \pi_1(C)\to \hat \pi_1(Y)/H']$
is nontrivial. This proves (b).
For the second statement, let $\Gamma = \rho(\pi_1(X))$.
By a well known theorem of Malcev (see \cite{magnus}), $\Gamma$ is residually finite. Therefore $\Gamma$ embeds into $\hat\Gamma$. Consequently for any $C\to X$, $\hat \pi_1(C)$ has an infinite image in $\hat \Gamma$, and therefore in  $\hat \pi_1(X)/\overline{\ker\rho}$. This proves (c).
Let $\pi:\tilde X_K\to X$ denote the projection.
  Suppose that $\pi^*\omega = d\alpha$, and suppose that  $f:C\to X$ is a nonconstant map from a curve with such that $\hat \pi_1(C)\to \hat \pi_1(X)$ is trivial.
  Then $f$ will lift to a map $\tilde f: C\to \tilde X_K$.  This implies that $f^*\omega= d\tilde f^*\alpha$ is an exact K\"ahler form. This is impossible. Therefore (d) holds.
   If $\pi_1(X)$ is residually finite, then $\tilde X = \tilde X_K$
  Furthermore, if $\pi_2(X)=0$, then $\pi_2(\tilde X)=0$. So by Hurewicz's theorem $H_2(\tilde X)=0$, which implies that $\pi^*\omega$ is exact. 
  Therefore (e) follows from (d) and \cite[cor. 0.4.C]{gromov}.
\end{proof}

\begin{cor}
 The examples discussed in the introduction are Koll\'ar-hyperbolic.
\end{cor}

First recall from the introduction that we define Kodaira fibrations inductively as projective manifolds admitting a smooth projective map to a curve of postive
genus with Kodaira fibrations as fibres.
We do not insist that the fibrations be nonisotrivial, so we can simply take products of curves,  however there are nontrivial examples as well \cite{bhpv, jab}.

\begin{proof}
Abelian varieties and locally symmetric spaces of noncompact type 
are aspherical, i.e. they have vanishing higher homotopy groups, with residually finite fundamental groups. Therefore
they are Koll\'ar-hyperbolic by (e). Alternatively, we can also use (c).
Therefore any normal variety with a finite map to either of these is Koll\'ar-hyperbolic by (b).

By definition a  Kodaira fibration $X$ is a bundle over a curve $C$ of genus $g>0$ with fibre $F$ a Kodaira fibration of lower dimension.
The homotopy exact sequence
$$\ldots \pi_i(F)\to \pi_i(X)\to \pi_i(C)\ldots$$
 and induction shows that $X$ aspherical.
 At the tail end, we have an exact sequence
 $$1\to \pi_1(F)\to \pi_1(X)\to \pi_1(C)\to 1$$
 Since the group on the right is residually finite (e.g. by Malcev's theorem), the residual  finiteness of $\pi_1(X)$ follows by induction. Therefore $X$ is Koll\'ar-hyperbolic by (e)
 
\end{proof}

There are many more examples of  Koll\'ar-hyperbolic varieties:

\begin{cor}\label{cor:finKhyp}
Given  a normal projective variety $X$, we can find a  Koll\'ar-hyperbolic variety $Y$ and  a finite map $Y\to X$. If $X$ is smooth,
$Y$ can be chosen smooth.
\end{cor}

\begin{proof}
Let $d = \dim X$.
Choose  a smooth Koll\'ar-hyperbolic variety $Z$ with $\dim Z= d$ (such as one of the examples already constructed).
Let  $Y\subset X\times Z$ be a complete intersection of $d$ very ample divisors in general position.
By Bertini, $Y$ is normal and smooth if $X$ is smooth.
 Both projections
$Y\to X$ and $Y\to Z$ are finite. Therefore  $Y$ is Koll\'ar-hyperbolic by   the proposition. 
\end{proof}

\begin{cor}\label{cor:alb}
If $X$ is smooth, then $(X, [\hat \pi_1(X), \hat \pi_1(X)])$ is Koll\'ar-hyperbolic if and only if the Albanese map is finite over it image.
\end{cor}

\begin{proof}
 One direction has essentially been proved (it follows form part (b) of the proposition). Suppose $(X, [\hat \pi_1(X), \hat \pi_1(X)])$ is Koll\'ar-hyperbolic.
 This easily implies that $H_1(C, \C)\to H_1(X,\C)$ is nonzero for any nontrivial map $C\to X$ from a smooth curve.
 If $alb:X\to alb(X)\subset Alb(X)$ was not finite, then we could find possibly singular curve $C'\subset X$ which gets contracted by $alb$.
 The map from the normalization  $(C')^n\to X$ would induce the zero map on $H_1$, and this is a contradiction.
\end{proof}
\section{V-hyperbolic varieties}

Let $X$ be a  connected variety with universal cover $\tilde X$.
Given a tower of \'etale covers
$$
\xymatrix{
 \ldots \ar[r]&X_3\ar[r]\ar[rrd]^{\pi_3} & X_2\ar[r]\ar[rd]^{\pi_2} & X_1\ar[d]^{\pi_1} \\ 
 & &  & X
}
$$
 we will say that it is connected if each $X_n$ is connected, Galois if each $X_n\to X$ is Galois, and unbounded if $\deg \pi_n\to \infty$.
A connected  (Galois) tower $\pi_n: X_n\to X$ is  determined by  a sequence 
$$\hat \pi_1(X)\supseteq \Gamma^1\supseteq\ldots$$
of open (normal) subroups  such that $X_n= \tilde X/\Gamma^n\cap \pi_1(X) $. Open subgroups of a profinite group are closed, therefore
 $H=\bigcap \Gamma^i$ is closed. Conversely, any such tower arises this way. If $H\subset \hat \pi_1(X)$ is a closed normal
subgroup, we refer to a connected   tower as above, with $\bigcap \Gamma^i = H$, simply as an {\em $H$-tower.}
An $H$-tower is unbounded if and only if $H$ has infinite index.

We will make use of the following lemma in various arguments. 

\begin{lemma}\label{lemma:restHtower}
Let $f:Y\to X$ be  a morphism of connected varieties, $H\subset \hat \pi_1(X)$ a closed normal subgroup, and let $G= f^{-1}(H)\subset \hat \pi_1(Y)$. 
If $X_n\to X$ is  an $H$-tower,  then we can find a $G$-tower $Y_n\to Y$, such that  each $X_n\times_X Y$ decomposes, as an \'etale cover of $Y$, into a disjoint union of a  finite number of copies of $Y_n$.
\end{lemma}

\begin{proof}
 This follows from standard covering space theory.
\end{proof}

 Given a  projective connected but possibly reducible variety and a closed normal subgroup $H\subset \hat \pi_1(X)$,
 let us say that $X$ has the {\em vanishing property}  with respect to $H$, or that 
 $(X, H)$ has the {\em vanishing property}, if $H$ has infinite index and  for any $H$-tower $\pi_n:X_n\to X$, for any  $P\in Perv_{HM}(X)$ and $i\not=0$,
$$\lim_{n\to \infty} \frac{h^i(X_{n}, \pi_n^*P)}{\deg \pi_n} = 0$$
We will say that $X$ is {\em V-hyperbolic}  with respect to $H$, or that 
 $(X, H)$ has is {\em  V-hyperbolic }, if $X$ is normal and  $(X,H)$ has the vanishing property.
We will say that $X$ is  V-hyperbolic if $(X,H)$ is V-hyperbolic for some $H$.
The reason for the terminology is the following:

\begin{lemma}
If $(X,H)$ is  as above is V-hyperbolic, 
then $(X,H)$ is Koll\'ar-hyperbolic. 
\end{lemma}

\begin{proof}
 Let $(X,H)$ be V-hyperbolic. Suppose  $(X,H)$ is not Koll\'ar-hyperbolic.
 Then we can find a smooth projective curve $C$ and a nonconstant map $f:C\to X$ such that $\hat \pi_1(C)\to \hat \pi_1(X)/H$
 is trivial. Choose an $H$-tower $\pi_n:X_n\to X$ with $G_n = Gal(X_n/X)$. We can find a compatible family of lifts $f_n:C\to X_n$ of $f$.
 Let $P= f_*\Q_C[1]$. Then 
 $$\pi_n^*\Q_C[1] = \bigoplus_{g\in G_n} g_* f_{n*} \Q_{C}[1]$$
 Therefore
 $$\lim_{n\to \infty} \frac{h^{-1}(X_{n}, \pi_n^*P)}{\deg \pi_n} =  h^0(C, \Q)=1$$
\end{proof}

We conjecture  the converse.

\begin{conj}\label{conj:main}
Given a normal projective  variety and a closed normal subgroup $H\subset \hat \pi_1(X)$,  if $X$ is 
is Koll\'ar-hyperbolic with respect to $H$, then $X$ it is V-hyperbolic with  respect to $H$.
\end{conj}

Supporting evidence for this conjecture will be presented in the next few sections.  
In this section, we will record a number of consequences of V-hyperbolicity.
We start with the following.

\begin{prop}
 If $(X,H)$ is V-hyperbolic , then for any $H$-tower $\pi_n:X_n\to X$, for any  $P\in Perv_{MHM}(X)$ and $i\not=0$,
$$\lim_{n\to \infty} \frac{h^i(X_{n}, \pi_n^*P)}{\deg \pi_n} = 0$$
\end{prop}

\begin{proof}
 This follows immediately by induction on the length of  the weight filtration $W$.
\end{proof}



The following corollary is both a special case and a generalization of the vanishing  theorems of Gromov \cite{gromov},  Cao-Xavier \cite{cx},  and Jost-Zuo \cite{jz}.
We also recall that the Singer conjecture \cite[chap 11]{luck2} states that the universal cover of an aspherical manifold has vanishing $L^2$-cohomology except in the middle degree. By \cite{atiyah}, this implies the Hopf-Singer conjecture that $(-1)^{\dim X}\chi(X)\ge 0$ under the same assumptions. See also
proposition \ref{prop:chiPge0}.

\begin{prop}
 Let $X$ be  a $d$-dimensional  V-hyperbolic smooth projective variety. Then the space of $L^2$  harmonic $k$-forms on $\tilde X_K$, with respect to the pull back of a K\"ahler metric from $X$,
 is zero for $k\not= d$.
 If conjecture \ref{conj:main} holds, then  the Singer conjecture holds when $X$ is an aspherical projective manifold with residually finite fundamental group.
\end{prop}

\begin{proof}
We refer to \cite{luck2} for  the background on $L^2$-Betti numbers. Set $i= k-d$.
 L\"uck \cite{luck} showed that 
 $$b_k^{(2)}= \lim_{n\to \infty}  \frac{h^i(X_n,\pi_n^*\Q[d])}{\deg \pi_n}$$
 is the $k$th $L^2$-Betti number of $\tilde X_K$, i.e. the Von Neumann dimension of the space of $L^2$  harmonic $k$-forms. Since $b_k^{(2)}=0$ when $i\not=0$,
 this space must vanish. The last statement of the corollary follows from what has just been proved by proposition \ref{prop:khyp}.
 
\end{proof}

The next proposition was conjectured by Maxim, Wang and the author \cite{amw} for aspherical manifolds and for varieties
with nef cotangent bundles.  The last inequality  goes back to Gromov \cite{gromov} for K\"ahler-hyperbolic manifolds.

\begin{prop}\label{prop:mhm}
Let  $X$ be a $d$-dimensional V-hyperbolic projective manifold.
If $M\in MHM(X)$, then for any $p$
 $$\chi(X,Gr^p_F DR(M))\ge 0$$
and in particular
 $$(-1)^{d-p}\chi(\Omega_X^p)\ge 0$$ 
\end{prop}

\begin{proof}
Let $P$ be the perverse sheaf underlying  the mixed  Hodge module $M$.
By  \cite[thm 28.1]{schnell}  $H^i(Gr^p_F DR(\pi_n^*M) )$ is a subquotient of
 $H^i(X_n, \pi_n^*P) \otimes \C$. 
We can see that
$$Gr^p_F DR(\pi_n^*M)\cong \pi_n^* Gr^p_F DR(M)$$
Therefore, when $i\not=0$,
\begin{equation}\label{eq:hiGr}
 \lim_{n\to \infty} \frac{h^i(X_n, \pi_n^*Gr^ p_F DR(M) )}{\deg \pi_n} =0 
\end{equation}
by the theorem.
Let 
$$L = \liminf_n \frac{h^0(X_n, \pi_n^*Gr^ p_F DR(M) )}{\deg \pi_n} $$
After passing to a subsequence,
 $$ \frac{\chi(X_n, \pi_n^*Gr^p_FDR(M))}{\deg \pi_n}=\sum_i (-1)^i  \frac{h^i(X_n, \pi_n^*Gr^p_FDR(M))}{\deg \pi_n}\to L$$
By Hirzebruch-Riemann-Roch \cite{fulton}
$$
 \frac{\chi(X_n, \pi_n^*Gr^p_FDR(M))}{\deg \pi_n}= \chi(X,Gr^p_F DR(M))
 $$
and therefore  this must equal $L$.
 Since $L$ is nonnegative,  we obtain
  $$\chi(X,Gr^p_F DR(M))\ge 0$$
When $M=\Q^H[d]$, we get
$$(-1)^{d-p}\chi(\Omega_X^p)\ge 0$$ 
\end{proof}

\begin{cor}
 The  conjecture of  \cite{amw}  stated above is implied by conjecture \ref{conj:main}, when $X$ is an  aspherical projective variety with residually finite fundamental group.
 \end{cor}

\begin{proof}
 This follows from proposition \ref{prop:khyp} and the previous proposition.

\end{proof}
The next result extends the Arakelov inequalities of Eyssidieux \cite{eyss} (when $X$ is weakly K\"ahler-hyperbolic and $D=0$), and  Peters \cite{peters} and Jost-Zuo \cite{jz2} (when $X$ is 
curve). However, the results for curves in the last two papers are considerably stronger than ours.

\begin{prop}
 With the notation and assumptions of example \ref{ex:Vd}, when $X$ is V-hyperbolic we have the Arakelov inequality
\begin{equation*}
\begin{split}
& \sum_k (-1)^{d-k}\chi(\Omega_X^k(\log D)\otimes  Gr_{\bar F}^{p-k} \bar V)\\
&=\sum_k (-1)^{d-k}\chi(\mathcal{H}^k(\Omega_X^k(\log D)\otimes  Gr_{\bar F}^{p-k} V);\theta)) \ge 0 
\end{split}
\end{equation*}
The sheaves appearing the last line are the cohomology of the complex appearing in example \ref{ex:Vd}.
\end{prop}

\begin{proof}
 This is a special case of the previous proposition.
\end{proof}

The next proposition, for arbitrary perverse sheaves, was proved by  Deng and Wang for Koll\'ar-hyperbolic manifolds with an (almost) faithful representation into a general linear group, and by Wang and the author \cite{aw}  with some additional assumptions. The methods are  completely different from what is used here.

\begin{prop}\label{prop:chiPge0}
 If $X$ is a  V-hyperbolic normal  projective variety and $P\in Perv_{MHM}(X)$, then 
 $$\chi(X, P)\ge 0$$
 If $X$ is also smooth of dimension $d$, then
 $(-1)^{d}\chi(X)\ge 0$. 
\end{prop}

\begin{proof}
The proof is very similar to  the proof of proposition \ref{prop:mhm}. In outline, using the lemma given below and the theorem, we have
 $$\chi(X, P) = \frac{ \chi(X_n,\pi_n^*P)}{\deg \pi_n}  \to \liminf \frac{h^0(X_n, \pi_nP)}{\deg \pi_n}\ge 0$$
 When $X$ is smooth, $\Q_X[d]\in Perv_{HM}(X)$.
\end{proof}

\begin{lemma}\label{lemma:eulercov}
 If $P\in D_c^b(X)$ is an object in the bounded constructible derived category,
 $$\chi(X_n, \pi_n^* P) = (\deg \pi_n ) \chi(X, P)$$
\end{lemma}

\begin{proof}
Choose a finite open cover  $\{U_i\}$ of $X$, such that each intersection  $U_{i_0\ldots i_m} = U_{i_0}\cap \ldots  \cap U_{i_m}$ is simply connected.
Then the pair $(\pi_n^{-1}U_{i_0\ldots i_m},  \pi_n^* P)$ will decompose into disjoint  union of $\deg \pi_n$  copies of $(U_{i_0\ldots i_m},   P)$. Therefore
\begin{equation*}
\begin{split}
 \chi(X_n, \pi_n^* P) &= \sum_m (-1)^m\chi(\pi_n^{-1}U_{i_0\ldots i_m}, \pi_n^* P)\\
&=\sum_m (-1)^m (\deg \pi_n) \chi(U_{i_0\ldots i_m}, P)\\
&= (\deg \pi_n) \chi(X, P)
\end{split}
\end{equation*}
\end{proof}

\section{Vanishing for curves }

In this section, we analyze the one dimensional case.
Let $C$ be a smooth projective curve of genus $g= g(C)$.  Any $P\in Perv_{HM}(C)$, or more generally a semisimple perverse sheaf, is a direct sum of a skyscraper sheaf and a sheaf of the form $P= IC(L):= j_*L[1]$, where $L$ is a local system defined on $j:U\to C$. In the second case, we will say that $P$ has full support, and we define $r(P)$ to be the rank of $L$ and
  $s(P) = \#(C-U)$, where $U$ is chosen to be maximal. If $P$ is a skyscraper sheaf, we set $r(P)=0$.

\begin{lemma}\label{lemma:bnd}
 Let $P\in Perv(C)$ be simple. If $r(P)=0$, then
 $$h^{-1}(C,P)= h^1(C,P)=0$$
 $$ h^0(C,P)=1$$
 If $r(P)>0$, then
 $$h^{-1}(C,P)\le r(P), h^1(C,P)\le r(P)$$
 $$h^0(C, P) \le (2g+s(P)) r(P)$$

\end{lemma}

\begin{proof}
Since $P$ is simple, it is either a skyscraper sheaf of the form $\Q_x$, or $P= j_*L[1]$, where $L$ is a local system defined on $j:U\to C$.
In the first case, we see immediately that 
$$h^{-1}(C,P)= h^1(C,P)=0, h^0(C,P)=1$$
In the second case,
$$H^{-1}(C,P) = L_x^{\pi_1(U,x)}$$
so that 
$$h^{-1}(C,P)\le r(P)$$
We have
$$ h^1(C,P)\le r(P)$$
by duality.
The distinguished triangle
$$ j_*L[1] \to Rj_* L[1]\to R^1j_*L\xrightarrow{[1]}$$
gives an injection
$$H^0(C,P)\hookrightarrow H^1(U, L)$$
Therefore
$$h^{0}(C,P)\le h^1(U, L) \le -\chi(U, L)+2r(P)=(2g+s(P)) r(P)$$
\end{proof}

The one dimensional case of  conjecture \ref{conj:main}  is easy, and it follows from next proposition.
It will be important to allow singular curves.

\begin{prop}\label{prop:curve}
 Let $C$ be a possibly singular and possibly reducible connected projective curve. 
 Let $H\subset \hat \pi_1(C)$ be a closed normal subgroup such that for any connected component $\tilde C^k$ of the normalization of $C$,
 the preimage of $H$ in $\hat \pi_1(\tilde C^k)$ has infinite index.
  If $\pi_n: C_n\to C$ is an  $H$-tower of \'etale covers, and $P\in Perv(C)$, then
 $$\lim_{n\to \infty} \frac{h^i(C_n, P)}{\deg \pi_n} = 0$$
 for $i \not= 0$. In particular, $(C,H)$ has the vanishing property.
\end{prop}

\begin{proof}
Since $Perv(C)$ is artinian,
 we can assume that $P$ is simple. Then $P$ is either a skyscraper sheaf, or $P= p_*j_*L[1]$, where $L$ is a local system defined on some
  $U\xrightarrow{j} \tilde C^k\xrightarrow{p} C$.
 In the first case, $H^{-1}(P) = H^1(P)=0$. Let us consider  the second case.
By lemma \ref{lemma:restHtower} the fibre product $C_n\times_C {\tilde C^k} \to \tilde C^k$ decomposes into a disjoint union of say $r_n$ copies  of a connected \'etale cover $\tilde C_n^k\to \tilde C^k$ of degree $d_n$. Then $\deg \pi_n = d_n r_n$. Since the preimage of $H$ in $\hat \pi_1(\tilde C^k)$ has infinite index, $\tilde C_n^k\to \tilde C^k$ is unbounded i.e.  $d_n\to \infty$.
 Therefore
 $$ \frac{h^{-1}(C_n, \pi_n^*P)}{\deg \pi_n} \le \frac{h^{-1}(\tilde C_n^k, \pi_n^*P)}{d_n}  \le\frac{r(\pi_n^*P)}{d_n}=  \frac{r(P)}{d_n} \to 0$$
 By duality, we get a similar vanishing statement for $h^1$ (or we could just repeat the same argument).
 

\end{proof}

\begin{cor}
 A curve is V-hyberbolic if and only if it has positive genus.
\end{cor}

\section{Finite covers}

\begin{lemma}\label{lemma:genfin}
 If $f:X\to Y$  is a proper surjective map of normal quasiprojective varieties, then $f( \pi_1(X))\subseteq  \pi_1(Y)$ has finite index, and  $f( \hat \pi_1(X))\subseteq \hat  \pi_1(Y)$
is open. 
 
 \end{lemma}

\begin{proof}
By taking a complete intersection of hyperplane sections, we can find a closed normal subvariety $X'\subset X$ such that $f' = f|_{X'}$ is generically finite.
Since $f(\pi_1(X))$ (or  $f(\hat \pi_1(X))$) contains $f'(\pi_1(X))$ (or  $f'(\hat \pi_1(X))$), there is no loss of generality in assuming that $f$ is generically
finite.
We can construct a commutative  diagram
 $$
 \xymatrix{
 U\ar[r]^{\iota}\ar[d]^{\pi} & \tilde X\ar[r]^{p}\ar[d]^{\tilde f} & X\ar[d]^{f} \\ 
 V\ar[r]^{\iota'} & \tilde Y\ar[r]^{p'} & Y
}
 $$
 where the arrows labeled $p,p'$ are resolutions of singularities, $\iota,\iota'$ are open immersions whose complements are divisors with simple normal crossings, and $\pi$ is an \'etale cover.  Since the maps $p,p'$ are surjective with connected fibres, they induce surjections on $\pi_1$. If $U = \tilde X- D$, then $\pi_1(\tilde X)$ is the quotient of $\pi_1(U)$ by the normal subgroup generated by loops about components of $D$. In particular, $\pi_1(U)\to \pi_1(\tilde X)$ is surjective. Similarly for $\iota'$. Therefore we have  a commutative diagram
 $$
 \xymatrix{
 \pi_1(U)\ar[r]\ar[d]^{\pi} & \pi_1(X)\ar[d]^{f} \\ 
 \pi_1(V)\ar[r] & \pi_1(Y)
}
$$
where the horizontal maps are surjective and $\im \pi$ has finite index because $\pi$ is an \'etale cover. 
This implies that $f( \pi_1(X))\subseteq  \pi_1(Y)$ has finite index, and  therefore that $f( \hat \pi_1(X))\subseteq \hat  \pi_1(Y)$ is open.
\end{proof}

\begin{thm}\label{thm:finVhyp}
  Let  $X$ be  V-hyperbolic (with respect to $H$).  Let $Y$ be a normal projective variety and let $f:Y\to X$ be a map.
\begin{enumerate}
 \item[(a)] Suppose $f:Y\to X$ is finite, or more generally suppose
   that $f$ is finite over $f(Y)$ and $f(\pi_1(Y))\subseteq \pi_1(X)$ has finite index. Then $Y$ is   
  V-hyperbolic (with respect to the preimage of $H$). 
\item[(b)] Let $V$ be a VHS on a Zariski open subset of $Y$. If $f$ is surjective, then for any $f^{-1}(H)$-tower $\pi_n:Y_n\to Y$,
$$\lim_{n\to \infty} \frac{h^i(Y_n, \pi_n^* S_Y(V))}{\deg \pi_n} = 0$$
for $i> \dim Y - \dim X$.
\item[(c)] Let $f$ be generically finite. If  $V$ is as in (b), then  $\chi(S_Y(V))\ge 0$. In particular, $\chi(Y, \omega_{Y}) \ge 0$ when $Y$ has rational singularities.
\end{enumerate}
  \end{thm}

\begin{proof}
Let $G\subset \hat \pi_1(Y)$ denote the preimage of $H$. Then $(Y, G)$ is Koll\'ar-hyperbolic by proposition \ref{prop:khyp}.
 Therefore $G$ has infinite index. Let $Y_n \to Y$ be a $G$-tower, and let
 $\Gamma^n \subset \hat \pi_1(Y)$ be the corresponding chain of open subgroups. Suppose
 that $f(\pi_1(Y))\subseteq \pi_1(X)$ has finite index either  by lemma \ref{lemma:genfin} or because
 we assumed it in  (a). Then $\Xi^n= f(\Gamma^n)\subset \hat \pi_1(X)$ is a chain of open subgroups. This corresponds to an $H$-tower $X_n\to X$.
 We have that $\Gamma^n\subseteq f^{-1} \Xi^n$, and we claim that these are equal. Since $\hat \pi_1(Y)/G\to \hat \pi_1(X)/H$ is injective, we can identify
 $$\Gamma^n/G =  (\Xi^n/H) \cap (\hat \pi_1(Y)/G)= f^{-1} \Xi^n/G$$
 Therefore $\Gamma^n= f^{-1} \Xi^n$ as claimed. Therefore we have a cartesian diagram
 $$
\xymatrix{
 Y_n\ar[r]^{\pi_n}\ar[d]^{f_n} & Y\ar[d]^{f} \\ 
 X_n\ar[r]^{p_n} & X
}
 $$
 By proper base change,  $\R f_{n*}\pi_n^*\F = p_n^* \R f_* \F$, for any bounded complex of sheaves $\F$ on $Y$. Therefore
\begin{equation}\label{eq:hiYneqXn}
  \frac{h^i(Y_n, \pi_n^* \F)}{\deg \pi_n} =  \frac{h^i(X_n, p_n^* \R f_*\F)}{\deg p_n}
\end{equation}

Suppose that the assumptions of (a) hold. In particular $f$ is finte over its image.
This implies that if  $P\in Perv_{HM}(Y)$, then $f_*P\in Perv_{HM}(X)$.  Finiteness, \eqref{eq:hiYneqXn} and the $V$-hyperbolicity of $X$  implies
 $$  \frac{h^i(Y_n, \pi_n^* P)}{\deg \pi_n} =  \frac{h^i(X_n, p_n^*  f_* P)}{\deg p_n}\to 0$$
  when $i\not= 0$. This proves (a).
 
 Let us turn to (b). Let $r = \dim Y -\dim X$ be   the dimension of the generic fibre.
  By \cite{saito3},
\begin{equation}\label{eq:decompS}
\R f_* S_Y(V) = \bigoplus_{i=0}^m S_X(V^i)[-i]  
\end{equation}
where $V^i$ are variations of Hodge structure defined on a smooth  open $U\subseteq X$, and $m\ge 0$ is some integer.
It follows that $S_X(V^i) = R^i f_* S_Y(V)$.
By shrinking $U$ if necessary, we can assume the dimension of all fibres of $f$ over $U$ are $r$. Therefore $S_Y(V^i)|_U=0$ for
$i> r$. Since $S_Y(V^i)$ is torsion free  (see example \ref{ex:Vd}), these sheaves vanish everywhere for $i>r$. Therefore, we can assume
that $m=r$. Consequently
$$\frac{h^i(Y_n, \pi_n^*S_Y(V) )}{\deg \pi_n} = \sum_{j=0}^r \frac{h^{i-j}(X_n, p_n^*S_X(V^i))}{\deg  p_n} = 0 $$
for $i>r$ by  \eqref{eq:hiYneqXn}, \eqref{eq:decompS}  and the fact that $X$ is $V$-hyperbolic.

When $f$  is generically finite, then  $\chi(S_Y(V))\ge 0$ follows from (b) (see the proof of proposition  \ref{prop:mhm}). Since  $S_X(\Q_Y^H[\dim Y]) = \omega_Y$ when $Y$ is smooth, 
we obtain $\chi(Y, \omega_{Y}) \ge 0$ in this case. When $Y$ has rational singularities, we get the same inequality by replacing $Y$  by a resolution of singularities.
 \end{proof}

\begin{rmk}
 Part (c) of the theorem gives evidence for a conjecture of Koll\'ar \cite[18.12.1]{kollar2} that $\chi(\omega_Y(V))\ge 0$ when $Y$ has generically large
 algebraic fundamental group. We will not recall what this means, but note that any variety satisfying the assumptions of (c) will satisfy the generic largeness condition.
Previous results in this direction were given by Jost-Zuo \cite{jz} and Deng-Wang \cite{ dw2}.
\end{rmk}
 
 \begin{cor}\label{cor:etaleVP}
Suppose that $(X,H)$ has the vanishing property. If $Y$ is an \'etale cover corresponding to
an open subgroup  $H\subset \hat \pi_1(Y)\subset \hat \pi_1(X)$, then $(Y,H)$ has the vanishing property.
%
%
\end{cor}

Although this follows immediately from the theorem, we give an easier direct proof.
\begin{proof}
Let $f:Y\to X$ be the structure map. Choose an $H$-tower $\pi_n: Y_n\to Y$, then
$p\circ \pi_n: Y_n \to X$ is an $H$-tower of $X$.
If $P\in Perv_{HM}(Y)$, $f_*P\in Perv_{HM}(X)$ and $P$ is a summand of $f^*f_*P$.
Therefore given $i\not=0$,
$$ \frac{h^i(Y_n, \pi_n^*P)}{\deg \pi_n} \le (\deg f) \frac{h^i(Y_n, \pi_n^*f^*f_*P)}{\deg f\circ  \pi_n}\to 0$$

\end{proof}

%
%
%
%
%

\section{Abelian varieties}

\begin{thm}\label{thm:av}
 Abelian varieties are V-hyperbolic.
\end{thm}

The proof will require  the following facts starting with the an elementary lemma.

\begin{lemma}\label{lemma:lattice}
Let $\Lambda$ a lattice with a  sequence of finite index subgroups 
$$\Lambda \supset \Gamma^1\supset \ldots$$
such that $\bigcap \Gamma^i = 0$. If $p: \Lambda\to \Xi$ is a quotient lattice such that
$ \operatorname{rank} \Xi <  \operatorname{rank} \Lambda$,
then
$$\frac{[\Xi : p(\Gamma^n)]} {[\Lambda:\Gamma^n] }\to 0$$ 
\end{lemma}

\begin{proof}
Let $\ell = \operatorname{rank} \Lambda$ and $ k= \operatorname{rank} \Xi$.
 We can identify $\Lambda = \Z^\ell$, $\Xi = \Z^k$ where $p$ is the projection onto the first $k$ factors.  By elementary divisor theory, we can 
 also assume that 
 $\Gamma^n = d_1(n)\Z\oplus \ldots \oplus d_\ell(n) \Z$ where each $d_i(n)$ is a sequence of integers dividing each successor and going to $\infty$.
 Then 
 $$ \frac{[\Xi : p(\Gamma^n)]} {[\Lambda:\Gamma^n] }= \frac{1}{d_{k+1}(n)\ldots d_\ell(n)} \to 0$$
\end{proof}

Next, we have the following result of Sarnak and Adams \cite[prop 1.6]{sa}.

\begin{prop}\label{prop:sa}
 Let $S$ be a real torus.
 If $W \subset S$ is an (real) algebraic subset, then there exists a finite number  of subtori $S_i\subset T$ and torsion points $s_i\in S_{tors}$
 such that  $s_i+S_i\subseteq W$ and
 $$ S_{tors}\cap W = S_{tors}\cap \bigcup_i (s_i+S_i)$$
\end{prop}

Let us  set up the notation for the corollary. Let
$$\Lambda \supset \Gamma^1\supset \ldots$$
be as in lemma \ref{lemma:lattice}. 
The Pontryagin dual $S = Hom(\Lambda, U(1))$ is a torus.
Then $G_n = \Lambda/\Gamma^n$ forms an inverse system
 $$\ldots G_3\to  G_2\to G_1$$
The duals $T_n = Hom(G_n, U(1))$  form a chain
$$T_1\subset T_2\subset T_3\ldots \subset  S_{tors}$$
such that $\bigcup T_n = S_{tors}$.

\begin{cor}\label{cor:sa}
With the above notation, let $W\subset S$ be a proper algebraic subset, then
 $$\frac{|T_n\cap W|}{|G_n|}\to 0$$
\end{cor}

\begin{proof}
 By the proposition, it suffices to assume that $W=s+S_1$, where $s\in S_{tors}$ and $S_1$ is a subtorus.
 Then $\Xi = Hom_{cont}(S_1, U(1))$ is quotient of $\Lambda$.
 Since $s\in T_n$ for almost all $n$, we can assume that $W=S_1$.
We have an isomorphism
$$T_n\cap S_1 \cong  Hom(\Xi/p(\Gamma^n), U(1))$$
Therefore the corollary follows from lemma \ref{lemma:lattice}.
\end{proof}

Finally we have the generic vanishing theorem of Popa and Schnell \cite[cor 2.7]{ps}.

\begin{thm}\label{thm:ps}
 Let $M$ be a Hodge module on an abelian variety $A$. Then for each $i>0$ and $p$, the set
 $$V^i(Gr^p_F DR(M))  = \{L\in \hat A\mid H^i(A, Gr^p_F DR(M) \otimes L)\not=0\}$$
 is proper Zariski closed subset of the dual variety  $\hat A = Pic^0(A)$.
\end{thm}

\begin{proof}[Proof of theorem \ref{thm:av}]
 Let $A$ be an abelian variety.
 Let $\Lambda = H_1(A,\Z)$, and let $$\Lambda \supset \Gamma^1\supset \ldots$$
 a  sequence of finite index subgroups  such that $\bigcap \Gamma^i = 0$. This gives a $0$-tower of \'etale covers
 $\pi_n: A_n\to A$. The Galois group $G_n= Gal(A_n/A) \cong \Lambda/\Gamma^n$.
 We can identify $\hat A$, as a real torus, with the Pontryagin dual of $\Lambda$ by sending a character $\chi\in Hom(\Lambda, U(1))$ to the line
 bundle $L_\chi$ with monodromy $\chi$. Therefore the character group
$T_n = Hom(G_n, U(1))$ embeds into $\hat A_{tors}$. 
 The group $G_n$ acts on $\pi_{n_*}\OO_{A_n} $. We can decompose it into isotypic components
 $$\pi_{n_*}\OO_{A_n} = \bigoplus_{\chi\in T_n } L_\chi$$

 Choose a Hodge module $M$ on $A$ and an integer $p$,
 and set
 $$V_p^i = V^i(Gr^p_F DR(M))$$
By the projection formula
$$H^i(A_n, \pi_n^* Gr^p_F DR(M)) \cong  \bigoplus_{\chi\in T_n }H^i(A,  Gr^p_F DR(M)\otimes L_\chi)$$
Therefore
\begin{equation*}
\begin{split}
 h^i(A_n, \pi_n^* Gr^p_F DR(M) &= \sum_{\chi\in T_n }h^i(A,  Gr^p_F DR(M)\otimes L_\chi)\\
 &\le m_p|T_n \cap V_p^i|
\end{split}
\end{equation*}
where $m_p$ is the maximum value attained by $h^i(A,  Gr^p_F DR(M)\otimes L)$ as $L$ varies over $V^i_p$.
Let $P$ be the perverse sheaf underlying $M$.
By theorem \ref{thm:strictness}, it follows that
$$\frac{h^i(A_n, \pi_n^* P)}{\deg \pi_n}\le \sum_p m_p\frac{|T_n \cap V_p^i|}{|G_n|}$$
By corollary \ref{cor:sa} and theorem \ref{thm:ps}, the right side goes to $0$ as $n\to \infty$ when $i>0$.
Therefore 
$$\lim_{n\to \infty}\frac{h^i(A_n, \pi_n^* P)}{\deg \pi_n} = 0$$
when $i>0$.
We deduce the vanishing of the same expression for $i<0$ by Poincar\'e-Verdier duality
$$H^i(X, P) \cong H^{-i}(X, DP)^*$$
and the fact  that $D:Perv(X)\to Perv(X)$  preserves $Perv_{HM}(X)$ by \cite[\S 2, \S5]{saito}.
\end{proof}

\begin{cor}
  If $f:X\to A$ is a map from a normal projective  variety to an abelian variety which is finite over its image, then $X$ is
  V-hyperbolic.
\end{cor}

\begin{proof}
 The image $f(H_1(X,\Q))\subset H_1(A,\Q)$ is a sub Hodge structure. Therefore it corresponds up to isogeny to a sub abelian variety of $B\subseteq A$.
 After replacing $A$ by $B$, we can assume that $f(H_1(X,\Z))$ has finite index. The  previous theorem together with theorem \ref{thm:finVhyp} implies the corollary.
\end{proof}

\begin{cor}
 Given  a normal projective variety $X$, we can find a  V-hyperbolic variety $Y$ and  a finite map $Y\to X$. If $X$ is smooth,
$Y$ can be chosen smooth.
\end{cor}

\begin{proof}
The argument is identical to the proof of  corollary  \ref{cor:finKhyp}, where $Z$ should be taken to be an abelian variety.
\end{proof}

 \section{The fibration theorem}

Our main theorem, that we call the fibration
theorem, gives evidence for conjecture \ref{conj:main}.


%
%
%

\begin{thm}\label{thm:key}
Let  $f: X\to C$ be a surjective map, with connected fibres, of a normal projective variety onto a curve of  genus $g>0$.
  Let $H\subset \hat \pi_1(X)$ be a closed normal subgroup. Suppose that
  
  \begin{enumerate}
\item[(a)]  the  image  $f (H)\subset \hat \pi_1(C)$, which is closed and normal, has infinite index,

\item[(b)] for all $y\in C$, the preimage $H_y$ of $H$ in   $\hat \pi_1(X_y)$ has infinite index in $\hat \pi_1(X_y)$, 

  \item[(c)]and for all $y\in C$, $(X_y, H_y)$ has the vanishing property.
\end{enumerate}
  Then $X$ is V-hyperbolic with respect to $H$.
\end{thm}

Of course (b) is redundant since it is included in (c), but we want to make the assumption  explicit. The proof will be given after establishing some basic lemmas. We fix the following notation for the remainder of this section.
Let  $f:X\to C$ and  H   be as in the  theorem, and let $\pi_n:X_n\to X$ be an $H$-tower.

\begin{lemma}\label{lemma:tower}

We have an unbounded connected tower of  \'etale covers
$$\ldots C_2\to  C_1\to C $$
For  each $n$, we have  a commutative diagram
 $$
\xymatrix{
 X_n\ar[r]_{q_n}\ar[rd]^{f_n}\ar@/^/[rr]^{\pi_n} & X_n'\ar[r]_{p_n'}\ar[d]^{g_n} & X\ar[d]^{f} \\ 
  & C_n\ar[r]^{p_n} & C
}
$$
such that all horizontal maps are \'etale,   the square is cartesian, and each fibre of
$f_n$ is  connected  with the vanishing property with respect to the preimage of $H$.
Furthermore
$$d_n = \deg p_n\to \infty$$
$$e_n= \deg q_n\to \infty$$
\end{lemma}

\begin{proof}
Let $X_n$ correspond to the sequence of open subgroups
$$\hat \pi_1(X)\supseteq \Gamma^1 \supseteq\ldots$$
Then the sequence of subroups
$$\Xi^n = f(\Gamma^n)\subset \hat \pi_1(C)$$
are open because $f:\hat \pi_1(X)\to \hat \pi_1(C)$ is surjective.
Therefore $\Xi^n$ gives rise to an  $f(H)$-tower $C_n'\to C$. 
The  \'etale cover  of  $X$ corresponding to $f^{-1}(\Xi^n)$ can be identified with $X\times_C C_n'$.
The inclusion $\Gamma_n\subseteq f^{-1}(\Xi^n)$, gives rise to an \'etale map $q_n':X_n\to X_n\times_C C_n$. Set $f_n'$ to the composite of the projection
$g_n: X_n'\to C_n$ with $q_n$. Stein factor $f_n'$ to obtain $X_n\to C_n \to C_n'$. 
Setting $X_n' = X\times_C C_n$, we see that $q_n'$ factors
through an \'etale map $q_n: X_n\to X_n'$. Thus we have constructed the diagram above.
By assumption (a) of  theorem \ref{thm:key}, $C_n'\to C$ is unbounded, therefore $C_n\to C$ is unbounded, i.e.
$$d_n = \deg p_n\to \infty$$
Similarly assumption (b) of  theorem \ref{thm:key} implies that the tower $X_{n,y_n}\to X_{y}$ is unbounded for $y\in C$ and a family of points $(y_n)\in \varprojlim p_n^{-1}(y)$.
Therefore
$$e_n= \deg q_n\to \infty$$
 The fibres of $f_n$ have the vanishing property by (c) and corollary \ref{cor:etaleVP}.
\bigskip
%

%
\end{proof}

%
%

We will use the notation from the  lemma \ref{lemma:restHtower} for the remainder of this section. 
Saito's version of the decomposition theorem \cite[cor. 3 of intro.]{saito}  implies that if $P\in Perv_{HM}(X)$, then
\begin{equation}\label{eq:decomp}
\R f_* P = \bigoplus_k {}^p R^k f_* P [-k]
\end{equation}
where
$${}^p R^k f_* P = {}^p \mathcal{H}^k( \R f_* P) \in Perv_{HM}(C)$$
Let us decompose this as
$${}^p R^k f_* P = F^k \oplus S^k$$
where $S^k$ is a finite sum
$$S^k = \bigoplus_{y\in C} S^k_y$$
with $S_y^k$ is supported on $y$,  and $F^k$ has full support.
Define
$$P_n = \pi_n^* P$$
$${}^p R^k f_{n*} P_n = F_n^k \oplus S_n^k$$
$$S^k_n = \bigoplus_{y\in C} S^k_{n,y}$$
as above. Then the decomposition theorem implies:

\begin{lemma}\label{lemma:decomp}
With the above notation,
 $$h^k(X_n, P_n) =  \sum_{i=-1}^1 h^{i}(C_n,  F_n^{k-i} \oplus S_n^{k-i})$$
\end{lemma}

We will need the following uniformity property.

\begin{lemma}\label{lemma:uniform}
With the above notation,
 there exists a nonempty Zariski open set $U\subseteq C$, such that each $S_n^k$ is supported on $p_n^{-1}(C-U)$.
 For each pair  $y,y'\in C_n$ lying in the same fibre over $y_0\in C-U$, there exists neighbourhoods $y\in D, y'\in D'$ such that
 there exists compatible homeomorphisms $h:D\cong D'$, $\eta: f_n^{-1} D\cong f_n^{-1} D'$ such that $\eta^*P_n|_{D'} \cong P_n|_D$.
\end{lemma}

\begin{proof}
The complex $P$ is constructible with respect to a certain Whitney stratification $\{S_\dt\}$ of $X$ (see \cite{htt,verdier} for the relevant definitions). 
We can choose
a Zariski open set $U\subset C$ over which the  stratification becomes locally trivial over $U$, i.e. such
that  we can find a stratified space $(F, T_\dt)$, an open cover $\{U_i\}$ of $U$, and stratified  homeomorphisms $U_i\times (F, T_\dt) \cong (f^{-1}U_i, f^{-1}U_i\cap S_\dt)$
compatible with projection to $U_i$.
 Then it follows easily
that each $R^k f_* P_n $ is locally constant over $U$. Therefore $S_n^k$ is supported on $p_n^{-1}(C-U)$. 

For the last part, let $\tilde X_n\to X$ be the Galois closure of $X_n\to X$.  This corresponds to an open normal subgroup $\tilde \Gamma^n\subset \hat \pi_1(X)$.
Let $H= Gal(\tilde X_n/X_n) \subset G= Gal(\tilde X_n/X)$ denote
the Galois groups.
We  modify the construction of lemma \ref{lemma:tower} by defining $\tilde f_n:\tilde X_n\to \tilde C_n$ to be  the Stein factorization of the map to the \'etale cover of
$C$ corresponding to $f(\tilde \Gamma^n)$. We can see that $G$ acts to $\tilde f_n$, and $f_n = \tilde f_n/H$ and $f= \tilde f_n/G$. Let $p:\tilde C_n\to C_n$ denote
the projection.
Given $y_0\in C-U$, choose a lift $\tilde y\in \tilde C_n$ and a small disk $\tilde D\subset \tilde C$ around it. 
 Then $y = p(\tilde y) \in p_n^{-1}(y_0)$ and any other point in the fibre is given by $y'=p(g \tilde y)$, for $g\in G$.
 Setting $D= p(\tilde D)$ and $D' = p(g\tilde D)$, we can see that the action of $g$ defines homeomorphisms $D\cong D'$, $ f_n^{-1} D\cong f_n^{-1} D'$ 
 as claimed.
\end{proof}

Given $y\in C_n$, choose a small disk $D= D_y\subset C_n$  centered at $y$, and choose local coordinate  $t$ such that $t=0$
defines $y$.
Let ${}^p\psi_t$ and ${}^p\phi_t$ denote the  nearby and vanishing cycle
functors of \cite{deligne} 
shifted by $[-1]$, either on $D$ or its preimage in $X_n$. 
A theorem of Gabber tells us that these functors take perverse sheaves to perverse sheaves, and in fact 
they preserve $Perv_{HM}$ (essentially by the way   Saito  defines Hodge modules \cite{saito}).

\begin{lemma}\label{lemma:rbnd}
 With notation as above, we have that the rank
 $$r(F_n^k) = h^k(X_{n,y}, {}^p\psi_t P_n )$$
 $$\dim S^k_{n,y} \le   h^k(X_{n,y}, {}^p\phi_t P_n)$$
\end{lemma}

\begin{proof}
For simplicity, we treat the case when $n=1$. The argument is the same in general.
Let $j:U\to C$ be the inclusion of the set defined in lemma \ref{lemma:uniform}.
We can assume that $y\notin U$.
Since $F^k$ has full support,  $F^k = j_*L^k[1]$ for some local system $L^k$ on $U$.
Then  \eqref{eq:decomp}  gives
$$\R f_* P = \bigoplus_k  j_* L^{k+1}[-k] \oplus S^k[-k]$$
Therefore
\begin{equation}\label{eq:Lk}
 L^{k} =  j^*R^{k-1}  f_*P
\end{equation}
%
 This implies that
\begin{equation*}
\begin{split}
 r(F^k)  &= h^{k-1}(X_z, P)\\
  &= h^k(X_y, {}^p\psi_t P) 
\end{split}
\end{equation*}
 for some $z\in D-\{0\}$. (To get to the last line, it is  helpful to remember that ${}^p\psi_t$ has a shift by $-1$.)
From \eqref{eq:decomp} and \eqref{eq:Lk}, we obtain
 $$(\R f_*P)_y 
 = \bigoplus_k ((j_*j^* R^k f_*P)_y \oplus S_y^k)[-k]$$
 From which we can conclude that
 $$H^k(X_y, P) = (j_*j^* R^k f_*P)_y \oplus S_y^k = H^k(X_z, P)^{inv} \oplus S_y^k$$
 The first summand on the right is the part of cohomology of the fibre $X_z$ over $z\in D-\{0\}$ invariant under the  local monodromy $T$.
 This can be rewritten more canonically in terms of nearby cycles,  
\begin{equation}\label{eq:dimS}
 \dim S_y^k = h^k(X_y, P) - \dim H^{k+1}({}^p\psi_t P)^{inv}
\end{equation}

 Next consider the diagram \cite{deligne}
 $$
 \xymatrix{
 \iota^* P [-1]\ar[r] & {}^p\psi_t P\ar[r]^{can}\ar[rd]^{1-T} & {}^p\phi_t P\ar[r]^{+1}\ar[d]^{var} &  \\ 
  &  & {}^p\psi_t P & 
}
 $$
 where $\iota: X_y\to X$ is the inclusion. The top line forms a distinguished triangle, therefore we have an exact sequence
 $$H^{k}(X_y,{}^p\phi_t P) \to H^{k}(X_y, P)\to H^{k+1}(X_y,{}^p\psi_t P)\xrightarrow{can} H^{k+1}(X_y,{}^p\phi_t P) $$
 From the diagram above, we have
 $$\ker can \subseteq \ker (1-T)$$
 When combined with \eqref{eq:dimS}, we obtain
$$\dim S_y^k \le   h^k({}^p\phi_t P)$$

\end{proof}


\begin{lemma}\label{lemma:key}
 With the assumptions of  theorem \ref{thm:key} and lemma \ref{lemma:tower} in force,
 $$\frac{h^i(C_n, F_n^k\oplus S_n^k )}{\deg \pi_n}\to 0$$
 if $k<0$ or if $k=0$ and $i\not=0$.
\end{lemma}

\begin{proof}
First let $k<0$.
Choose $U$ as in lemma \ref{lemma:uniform}. Let $\sigma= \#(C-U)$, and
let $C-U=\{y_1,\ldots, y_\sigma\}$.
Pick a  point $y\in C$, and let $y(n)\in C_n$ be a point lying over it.
Choose a small disk $D_{y}$ centered at $y$, with local parameter $t$.
This can be identified with a small disk around $y(n)$. We let $t$ also
denote its pull back to the second disk. Let $Y= X_y$ and $Y_n=X_{n, y(n)}$ denote the fibres. 
With these identifications,
we get a tower of \'etale covers $\pi_n:Y_n\to Y$ each equipped with a perverse sheaf
${}^p\psi_t P_n \cong  \pi_n^*\,{}^p\psi_t P$  \cite[2.1.7.2]{deligne} in $Perv_{HM}$.  
Since we are assuming that $Y$ has the vanishing property with respect to  the preimage of $H$ (assumption (c) of theorem \ref{thm:key}),
\begin{equation}\label{eq:limY}
 \frac{h^k(Y_{n}, \pi_n^*Q)}{e_n}\to 0, \quad Q\in Perv_{HM}(Y)
\end{equation}
 By Riemann-Hurwitz and lemmas \ref{lemma:bnd}, \ref{lemma:uniform} and \ref{lemma:rbnd} 
\begin{equation*}
\begin{split}
  h^{0}(C_n, F_n^k) &\le (2g(C_n)+ s(F_n^k)) r(F^k_n)\\
   &\le  (2g(C_n)-2 +2d_n+ s(F_n^k)) h^k(Y_{n}, {}^p\psi_tP_n)\\
&\le (2g(C) + \sigma)d_n h^k(Y_{n}, {}^p\psi_t P_n)
\end{split}
\end{equation*}

Therefore since $\deg \pi_n= d_ne_n$,
$$0\le \frac{h^{0}(C_n, F_n^k)}{\deg \pi_n}\le (2g + \sigma) \frac{h^k(Y_{n},{}^p\psi_t P_n)}{e_n}$$
The right side converges to $0$ by \eqref{eq:limY}.

For each $y_j \in C-U$, let $p_n^{-1}(y_j) = \{\tilde y_{j,1},\ldots, y_{j, d_n}\}$ denote the fibre, and let $Y_{n,j} = X_{n, \tilde y_{j,1}}$.
Choose small coordinate disks $D_{y_j}$, with local parameters $t_j$, about each point $y_j$.
With similar identifications to those above,  lemmas \ref{lemma:uniform} and
 \ref{lemma:rbnd} shows that 
 \begin{equation*}
\begin{split}
  h^0(C_n, S_n^k) &=  \sum_{j=1}^\sigma\sum_{\ell=1}^d \dim S_{n, \tilde y_{j,\ell}}^k\\
  &\le d_n \sum_{j=1}^\sigma h^k(Y_{n,j},{}^p\phi_{t_j} P_n) 
\end{split}
\end{equation*}
 Therefore
$$\frac{h^0(C_n,  S_n^k )}{\deg \pi_n}\le \sum_{j=1}^\sigma \frac{h^k(Y_{n,j},{}^p\phi_{t_j} P_n)}{e_n}$$
This goes to zero by \eqref{eq:limY}. This proves the lemma when $i=0$ and $k<0$.

Now suppose that  $i=\pm 1$. With the same  set up as in  the first paragraph,  lemmas \ref{lemma:bnd} and \ref{lemma:rbnd} show that
$$\frac{h^i(C_n, F_n^k\oplus S_n^k )}{\deg \pi_n} = \frac{h^i(C_n, F_n^k)}{\deg \pi_n}\le \frac{h^k(Y_{n},{}^p\psi_t P_n)}{d_ne_n}$$
We claim that the expression on the right converges to 0 for all $k\le 0$,
When $k<0$, $h^k(Y_{n},{}^p\psi_t P_n)/{e_n}\to 0$ by \eqref{eq:limY}.  When $k=0$, we see that
$$\frac{h^0(Y_{n},{}^p\psi_t P_n)}{e_n}\to \frac{\chi(Y_n, {}^p\psi_t P_n)}{e_n} = \chi(Y_1, {}^p\psi_t P_1)$$
by lemma \ref{lemma:eulercov}.
Therefore
$$\frac{h^0(Y_{n},{}^p\psi_t P_n)}{d_ne_n}\to 0$$
\end{proof}

\begin{proof}[Proof of theorem \ref{thm:key}]
 Lemmas \ref{lemma:decomp} and \ref{lemma:key} imply that
 $$\frac{h^i(X_n, P_n)}{\deg \pi_n} \to 0$$
 when $i<0$. We deduce the vanishing for $i>0$ by Poincar\'e-Verdier duality. 
\end{proof}


We define  a generalized Kodaira fibration inductively as a Koll\'ar-hyperbolic variety $X$ which possesses a surjective morphism
$f:X\to C$ to a curve of positive genus, such that normalizations of irreducible components of all fibres are generalized Kodaira fibrations.
Examples include Kodaira fibrations as defined earlier, and surfaces with irrational pencils uniformized by the ball (which exist, see for example \cite{hirzebruch})

\begin{cor}
 A generalized Kodaira fibration is V-hyperbolic.
\end{cor}

\begin{proof}
 This follows from the theorem plus induction on dimension.
\end{proof}

\end{document}